\documentclass[a4paper,twoside]{amsart}

\usepackage{amssymb,amsfonts,amsmath}
\usepackage{latexsym}
\usepackage[cp1251]{inputenc}
\usepackage[russian,english]{babel}

\newcommand{\diam}{{\mathrm{diam}}}

\newcommand{\radius}{{\mathrm{rad}}}

\newtheorem{theo}{Theorem}[section]
\newtheorem{prop}[theo]{Proposition}
\newtheorem{cor}[theo]{Corollary}

\newcommand{\Trans} {{\mathrm{Trans}}}

\newcommand{\Prox}{{\mathrm{Prox}}}

\newcommand{\Eq}{{\mathrm{Eq}}}
\newcommand{\Equ}{{\mathrm{Equ}}}
\newcommand{\ToM}{{\mathrm{ToM}}}

\makeatletter
\@namedef{subjclassname@2010}{%
  \textup{2010} Mathematics Subject Classification}
\makeatother

\begin{document}

\title{On the Lyapunov numbers}

\author[S. Kolyada]{Sergi\v{\i} Kolyada}
\address{Institute of Mathematics, NASU \\ Tereshchenkivs'ka 3 \\ 01601 Kyiv, Ukraine}
\email{skolyada@imath.kiev.ua}

\author[O. Rybak]{Oleksandr Rybak}
\address{Institute of Mathematics, NASU \\ Tereshchenkivs'ka 3 \\ 01601 Kyiv, Ukraine}
\email{semperfi@ukr.net}


\begin{abstract}
We introduce and study the Lyapunov numbers -- quantitative measures of the sensitivity of a dynamical system $(X,f)$ given by a compact metric space $X$ and a continuous map $f:X \to X$.  In particular, we prove that for a minimal topologically weakly mixing system all Lyapunov numbers are the same.
\end{abstract}

\subjclass[2010]{Primary 37B05; Secondary 54H20, 37B25}

\keywords{Lyapunov numbers, minimal map, sensitive dynamical system, topologically weakly mixing map}

\maketitle

\section{Introduction}

Throughout this paper $(X,f)$ denotes a \textit{topological dynamical system}, where
$X$ is a compact metric space with metric $d$ and $f:X \to X$ is a continuous map.

The notion of sensitivity (sensitive dependence on initial conditions) was first used by Ruelle \cite{Ru}. According to  the works by Guckenheimer \cite{Gu}, Auslander and Yorke \cite{AuY} a dynamical system $(X,f)$ is called
{\it sensitive} if there exists a positive $\varepsilon$ such that for every $x\in X$ and every
neighborhood $U_x$ of $x$, there exist $y\in U_x$ and a nonnegative integer $n$ with $d(f^n(x),f^n(y))> \varepsilon$.

Recently several authors studied the different properties related
to sensitivity (cf. Abraham et al.\cite{ABC}, Akin and Kolyada
\cite{AK}, Moothathu \cite{Moo}, Huang et al. \cite{HLY}). The
following proposition holds according to \cite{AK}.

\begin{prop} Let $(X,f)$ be a topological dynamical system. The following conditions are
equivalent.
\begin{itemize}
\item[1.] $(X,f)$ is sensitive.
\item[2.] There exists a positive $\varepsilon$ such that for every $x\in X$ and every
neighborhood $U_x$ of $x$, there exists $y\in U_x$ with $\limsup_{n\to \infty}d(f^n(x),f^n(y))> \varepsilon$.
\item[3.] There exists a positive $\varepsilon$ such that in any opene\footnote{Because we so often have to refer to open, nonempty subsets, we will call such subsets \emph{opene}.}
$U$ in $X$ there are $x, y\in U$ and a nonnegative integer $n$ with $d(f^n(x),f^n(y))> \varepsilon$.
\item[4.] There exists a positive $\varepsilon$ such that in any opene $U\subset X$ there are $x,y\in U$ with $\limsup_{n\to \infty}d(f^n(x),f^n(y))> \varepsilon$.
\end{itemize}
\end{prop}

For a dynamical system $(X,f)$ a point $x \in X$  is Lyapunov stable if the dependence of the orbit upon
the initial position is continuous at $x$ (see \cite{AK}).  This is most easily defined using the
$f-$extension of the metric $d$:
\begin{eqnarray*}
d_{f}(x,y) =  \sup \  \{d(f^{n}(x),f^{n}(y)): n \geq 0 \}
\end{eqnarray*}
for $x,y \in X$. Clearly, $d_{f}$ is a metric on $X$ and
\begin{eqnarray*}
d_{f}(x,y)  =  \max[\ d(x,y),\ d_{f}(f(x),f(y))\ ].
\end{eqnarray*}
Using these metrics we define the \emph{diameter} and \emph{f-diameter} for $A \subset X$, the \emph{radius} and \emph{f-radius} for a neighborhood $U_x$ of a point $x \in X$
\begin{eqnarray*}
\diam(A) =\sup \ \{d(x,y) : x,y \in A \},
\diam_f(A) = \sup \ \{d_{f}(x,y) : x,y \in A \}, \\
\radius(U_x) =\sup \ \{d(x,y) : y \in U_x \},
 \radius_f(U_x) =
\sup \ \{d_{f}(x,y) : y \in U_x \}.
\end{eqnarray*}

The topology obtained from the metric $d_{f}$ is usually strictly coarser than the original $d$ topology.  When we use a term like ``open", we refer exclusively to the original topology.

A point $x \in X$ is called \emph{Lyapunov stable } if for every $\varepsilon > 0$ there exists a $\delta > 0$ such that $\radius (U_x) < \delta$ implies
$\radius_{f}(U_x) \leq \varepsilon$.  This condition says exactly that the sequence of iterates $\{f^{n} : n \geq 0 \}$ is equicontinuous at $x$. Hence, such a point is also called an \emph{equicontinuity point}.
We label associated point sets:
\begin{eqnarray*}
\Eq_\varepsilon(f)  = \bigcup \ \{U_x \subset X: U_x \mbox{ is a
neighborhood of a point $x$} \\ \mbox{ with } \radius_f(U_x) \leq
\varepsilon \} \mbox{ and } \Eq(f) = \bigcap_{\varepsilon > 0} \
\Eq_\varepsilon(f).
\end{eqnarray*}
As the label suggests, $\Eq(f)$ is the set of equicontinuity
points. If $\Eq(f) = X$, i.e. every point is equicontinuous, then
the two metrics $d$ and $d_{f}$ are topologically equivalent and
so, by compactness, they are uniformly equivalent. Such a system
is called \emph{equicontinuous}.  Thus, $(X,f)$ is equicontinuous
exactly when the sequence $\{f^{n} : n \geq 0 \}$ is uniformly
equicontinuous.

If the $G_{\delta}$ set $\Eq(f)$ is dense in $X$ then the system is called \emph{almost equicontinuous}. On the other hand, if $\Eq_\varepsilon(f) = \varnothing$ for some $\varepsilon > 0$ then it is the same that the system
shows \emph{sensitive dependence upon initial conditions} or, more simply, $(X,f)$ is \emph{sensitive}. We define
\begin{eqnarray*}
\mathbb{L}_{r}: = \sup\{\varepsilon: \mbox{ for every } x\in X
\mbox{ and every neighborhood } U_x \mbox{ of } x \mbox{ there }
\\ \mbox{ exist } y\in U_x  \mbox{ and a nonnegative integer }
n  \mbox{ with } d(f^n(x),f^n(y))> \varepsilon
 \}
\end{eqnarray*}
and call it \emph{the} (\emph{first}) \emph{Lyapunov number}.

It can happen that $\Eq_\varepsilon(f) \not= \varnothing$ for all positive $\varepsilon$ and yet still the intersection, $\Eq(f)$, is empty (see \cite{AK}).
This cannot happen when the system is transitive\footnote{We recall the definition in Section 3.} (Glasner and Weiss \cite{GW},
Akin et al \cite{AAB}).

\begin{theo}  Let $(X,f)$ be a topologically transitive dynamical system.  Exactly one of the following two cases holds.
\begin{enumerate}
\item[]{\bf Case i} \emph{(}$\Eq(f) \not= \varnothing$\emph{)}  Assume there exists an equicontinuity point for the system.  The equicontinuity points are exactly the transitive points, i.e. $\Eq(f) = \Trans(f)$, and the system is almost equicontinuous.  The map $f$ is a homeomorphism and the inverse system $(X,f^{-1})$ is almost equicontinuous.  Furthermore, the system is \emph{uniformly rigid} meaning that some subsequence of $\{ f^{n} : n= 0,1,... \}$ converges uniformly to the identity.
\item[]{\bf Case ii} \emph{(}$\Eq(f) = \varnothing$\emph{)}  Assume the system has no equicontinuity points.  The system is sensitive, i.e. there exists $\varepsilon > 0$ such that $\Eq_\varepsilon(f) = \varnothing$.
\end{enumerate}
\end{theo}

\begin{cor}  If $(X,f)$ is a minimal dynamical system then it is either sensitive or equicontinuous.
\end{cor}

Let us define
\begin{eqnarray*}
\Equ_\varepsilon(f) =  \bigcup \ \{ U \subset X : U \mbox{ is open with } \diam_f(U) \leq \varepsilon\}.
\end{eqnarray*}
Obviously  $\Eq(f)=\cap_{\varepsilon > 0} \ \Equ_\varepsilon(f)$,  and if $\Equ_\varepsilon(f) = \varnothing$ for some $\varepsilon > 0$ then the system
$(X,f)$ is sensitive (see also Proposition 1.1). Therefore, it is natural to define
\begin{eqnarray*}
\mathbb{L}_{d} : = \sup\{\varepsilon: \mbox{ in any opene }
U\subset X \mbox{ there exist } x,y\in U  \mbox{ and there is} \\
\mbox {a positive integer } n \mbox{ with }  d(f^n(x),f^n(y))>
\varepsilon \}
\end{eqnarray*}
and call it \emph{the second Lyapunov number}.

According to Proposition 1.1 we will define
\begin{eqnarray*}
\overline{\mathbb{L}}_{r} : = \sup\{\varepsilon: \mbox{ for every } x\in X \mbox{ and every open neighborhood } U_x \mbox{ of } x \\
\mbox{ there exists } y\in U_x \mbox{ with }
\limsup_{n\to \infty}d(f^n(x),f^n(y))> \varepsilon\}, \\
\overline{\mathbb{L}}_{d} : = \sup\{\varepsilon: \mbox{ in any
opene } U\subset X \mbox{ there exist } x,y\in U
\mbox{ with } \\
\limsup_{n\to \infty}d(f^n(x),f^n(y))> \varepsilon \}.
\end{eqnarray*}

Sometimes it will be useful to use also the following notations $$\mathbb{L}_1:= \mathbb{L}_{r}; ~\mathbb{L}_2:= \mathbb{L}_{d};~ \mathbb{L}_3:= \overline{\mathbb{L}}_{r}; ~\mathbb{L}_4:= \overline{\mathbb{L}}_{d}.$$
So, various definitions of sensitivity, formally give us different Lyapunov numbers -- quantitative measures of these sensitivities.

In Section 2 we  prove that for a topological dynamical system $(X,f)$, it holds $\mathbb{L}_d \le 2 \overline{\mathbb{L}}_{r}$. In Section 3 we examine the equalities between the Lyapunov numbers for topologically
transitive systems and in Section 4 for weakly mixing  systems. In particular, we prove that for topologically weakly mixing minimal systems  all Lyapunov numbers are the same. Finally, in Section 5 we give some examples and open problems for Lyapunov numbers.

\medskip
\noindent {\bf Acknowledgements.}  We thank the anonymous reviewer for helpful remarks and suggestions.  The first author was supported by Max-Planck-Institut f\"ur Mathematik (Bonn); he acknowledges the hospitality of the Institute.

\section{A general inequality for the Lyapunov numbers}

Directly  from the definitions,  the following  inequalities hold
$$\mathbb{L}_d\ge \overline{\mathbb{L}}_{d} \ge
\overline{\mathbb{L}}_{r} \mbox{ and } \mathbb{L}_d\ge
\mathbb{L}_r \ge \overline{\mathbb{L}}_{r}.$$

\begin{prop} Let $(X,f)$ be a sensitive topological dynamical system. Then
 $\mathbb{L}_d \le 2 \overline{\mathbb{L}}_{r}$.
\end{prop}

\begin{proof} Let $\mathbb{L}_d$ be the second Lyapunov number
of $(X,f)$. Fix a (small enough) $\delta>0$, a point $x \in X$ and a neighborhood $U_x$ of
$x$.
Let $U_0=U_x$ and $n_0$ be the first positive integer, for which $\diam(f^{n_0}(U_0)) > \mathbb{L}_d-\delta$.
There exists a point $y_0 \in U_0$ such that
$d(f^{n_0}(x),f^{n_0}(y_0)) > (\mathbb{L}_d-\delta)/2$. Choose an opene $U_1$ with its closure contained in
$U_{0}$ such that $y_0 \in U_1$ and $\diam(f^m(U_1)) \le \delta/2$ for every non-negative integer $m \le n_0$.
Let $n_1$ be the first positive integer, for which $\diam(f^{n_1}(U_1)) > \mathbb{L}_d-\delta$. By the definition
of $U_1$, we clearly have $n_1>n_0$.

We define recursively opene sets $U_2, U_3, ... $ and positive
integers $n_2, n_3, ... $ as follows. Since $n_{k-1}$ is defined, there exists a point $y_{k-1} \in U_{k-1}$ such that
$d(f^{n_{k-1}}(x), f^{n_{k-1}}(y_{k-1})) > (\mathbb{L}_d-\delta)/2$. So we can choose an opene $U_k$ in $U_{k-1}$ such that $y_{k-1} \in U_k$ and $\diam(f^m(U_{n_k})) \le \delta/2$ for every non-negative
integer $m \le n_{k-1}$. Let $n_k$ be the first positive integer, for which $\diam(f^{n_k}(U_k)) > \mathbb{L}_d-\delta$.
As in the previous step, by the definition of $U_k$ we clearly have $n_k>n_{k-1}$.

If $y$ is a point of the nonempty intersection $\cap_k
\overline{U_{n_k}}$, then, obviously, $y \in U$ and $\limsup_{n
\to \infty}\ d(f^n(x),f^n(y)) \ge \mathbb{L}_d/2-\delta$.
\end{proof}

As a consequence of the inequalities at the beginning of  Section 2 and Proposition 2.1 we conclude that $\mathbb{L}_i \le 2 \mathbb{L}_j$ for any $i,j\in \{1,2,3,4\}$.

\section {Lyapunov numbers for transitive maps}

A topological dynamical system $(X,f)$ is called \textit{topologically transitive}, if for any pair of opene subsets
$U,V \in X$ ~ $$n_f(U,V):=\{n\in \mathbb{Z}_{+}: U \cap f^{-n}(V) \ne \varnothing \}$$ is infinite.  A point $x \in X$ is called a
\textit{transitive point} if its orbit $\{x,\ f(x),\ f^2(x),\ ...\}$
is dense in $X$. If $(X,f)$ is topologically transitive and $X$ is compact, then the set of
transitive points is a $G_{\delta}$-dense subset of $X$.

If every point of a dynamical system $(X,f)$ is transitive, then this system is called
\textit{minimal}. An $f$-invariant  closed subset $M \subset X$ is called \emph{minimal} if the orbit of any point of $M$ is dense in $M$ (in this case a point of $M$ is called \emph{minimal}, too).

For a dynamical system $(X,f)$, a point $x\in X$ and a set $U\subset X$ let $$n_f(x,U):= \{n\in \mathbb{Z}_{+}: f^n(x) \in U\}.$$
A point $x\in X$ is said to be \emph{recurrent} if for every neighborhood $U$ of $x$~ the set $n_f(x,U)$ is infinite.

A subset $S$ of $\mathbb{Z}_{+}$ is \emph{syndetic} if it has  bounded gaps, i.e. there is $N\in \mathbb{N}$ such that $\{i,i+1, ... , i+ N\} \cap S\not =\varnothing$ for every $i\in \mathbb{Z}_{+}$. $S$ is \emph{thick} if it contains arbitrarily long runs of positive integers, i.e. there is a strictly increasing subsequence $\{n_i\}$  such that $S \supset \cup_{i=1}^\infty \{n_i, n_i + 1, ... , n_i + i\}$.

Some dynamical properties can be interpreted by using the notions of syndetic or thick
subsets. For example, a classic result of Gottschalk states that $x\in X$ is a minimal point
if and only if $n_f(x,U)$ is syndetic  for any neighborhood $U$ of $x$, and a topological dynamical system $(X,T)$ is (topologically) weakly mixing (we recall the  definition in Section 4) if and only if $n_f(U,V)$ is  thick for any opene subsets $U,V$ of $X$ \cite{F1},\cite{F2}.

\begin{theo} Let $(X,f)$ be a sensitive topologically transitive dynamical system.
Then $\mathbb{L}_{d}=\overline{\mathbb{L}}_{d}$.
\end{theo}

\begin{proof} By the definition of $\mathbb{L}_{d}$, for any
$\varepsilon < \mathbb{L}_{d}$ and for any opene $U \in X$ there
are points $x,y \in U$ and a positive integer $n_0$ such that
$d(f^{n_0}(x),f^{n_0}(y))>\varepsilon$. Choose an arbitrary
(small) $\delta>0$. Let $U_x, U_y \subset U$ be neighborhoods of
$x$ and $y$ such that $\diam(f^{n_0}(U_x))<\delta$ and
$\diam(f^{n_0}(U_y))<\delta$. If $z \in U_x$ is a transitive
point, there is a positive integer $m$ for which $f^m(z) \in U_y$.
By the triangle inequality we have
$d(f^{n_0}(z),f^{n_0+m}(z))>\varepsilon-2\delta$.

Let $U_z$ be a neighborhood of $z$ such that $U_z \subset U_x$ and
$f^m(U_z) \subset U_y$. Then obviously
$\diam(f^{n_0}(U_z))<\delta$ and $\diam(f^{n_0+m}(U_z))<\delta$.
Since a sensitive system has no isolated points, $U_z$ is
infinite. Therefore, the orbit of the point $z$ visits $U_z$
infinitely many times. If $n_k$ is such that $f^{n_k}(z)\in U_z$,
then $f^{n_0+n_k}(z)=f^{n_0}(f^{n_k}(z)) \subset f^{n_0}(U_z)$ and
$f^{n_0+n_k+m}(z)=f^{n_0+m}(f^{n_k}(z)) \subset f^{n_0+m}(U_z)=
f^{n_0}(f^m(U_z)) \subset f^{n_0}(U_y)$. And so, by the triangle
inequality, $d(f^{n_0+ n_k}(z),f^{n_0+n_k +m}(z)) >
\varepsilon-2\delta$. From this we have $\overline{\mathbb{L}}_{d}
> \limsup_{n \to \infty} d(f^n(z),f^n(f^m(z))) \ge
\varepsilon-2\delta$. Since $\delta>0$ and $\varepsilon <
\mathbb{L}_{d}$ were chosen arbitrarily,
$\mathbb{L}_{d}=\overline{\mathbb{L}}_{d}$.
\end{proof}

A topologically transitive dynamical system $(X,f)$, where $X$ has
no isolated points, is called $\ToM$ if every point $x\in X$ is
either (topologically) transitive or minimal. $\ToM$ systems were
introduced by Downarowicz and Ye in \cite{DY}. Since we do not
require that both types are present (as in \cite{DY}), a minimal
system is also $\ToM$. If a $\ToM$ system is not minimal, then the
set of minimal points is dense in $X$ (because for a transitive,
but non-minimal system, the set of non-transitive points is dense
(see for instance \cite{KS})).

\begin{theo} Let $(X,f)$ be a sensitive $\ToM$ system.
Then $\mathbb{L}_{r}=\overline{\mathbb{L}}_{r}$.
\end{theo}

\begin{proof} Fix a point $x \in X$. Let $U_x$ be a neighborhood
of $x$ and let $\delta >0$. By the definition of $\mathbb{L}_{r}$,
there exist a point $y \in U_x$ and a positive integer $m$ such
that $d(f^m(x),f^m(y))>\mathbb{L}_{r}-\delta$. Take a neighborhood
$U_y\subset U_x$ of point $y$ such that $\diam f^m(U_y)<\delta$.

Now, if $x$ is a transitive point,  one can just repeat the idea of the proof of Theorem 3.1  for the proof of this case. If $x$ is not transitive, then is minimal. Since $(X,f)$ is $\ToM$, we can find a minimal point $z_1 \in U_y$ and therefore  $d(f^m(x),f^m(z_1))>\mathbb{L}_{r}-2\delta$.

Consider the direct product system $(\overline{\mathrm{Orb}_f(x)}
\times \overline{\mathrm{Orb}_f(z_1)},
f|_{\overline{\mathrm{Orb}_f(x)}} \times
f|_{\overline{\mathrm{Orb}_f(z_1)}})$. Let $M$ be a minimal subset
of this system. Then obviously $M\cap M_x\not = \varnothing$,
where $M_x:=\{(x,z): z\in \overline{\mathrm{Orb}_f(z_1)}\}$. Hence
there is a point $(x,z_2)\in U_x \times
\overline{\mathrm{Orb}_f(z_1)}$, which is minimal, and therefore
(uniformly) \emph{recurrent} for the map
$f|_{\overline{\mathrm{Orb}_f(x)}} \times
f|_{\overline{\mathrm{Orb}_f(z_1)}})$. Clearly, every point of the
form $(x,f^k(z_2)),~ k=0,1,2, ... $ will be uniformly recurrent
too. Since $z_1$ is minimal, we can take a positive integer $k$,
such that $z_3:=f^k(z_2)\in U_y$. Therefore, we have $\limsup_{n
\to \infty} d(f^n(x),f^n(z_3)) \ge \mathbb{L}_{r}-2\delta$. Since
$x$ and $\delta>0$ were chosen arbitrarily, we get
$\mathbb{L}_{r}=\overline{\mathbb{L}}_{r}$.
\end{proof}

As a corollary of the last two theorems we conclude that the equalities  $\mathbb{L}_{r}=\overline{\mathbb{L}}_{r}$ and  $\mathbb{L}_{d}=\overline{\mathbb{L}}_{d}$ hold for minimal dynamical systems. And what we can say about dynamical systems for which  $\mathbb{L}_{r}=\mathbb{L}_{d}$ holds?

\section {Lyapunov numbers for weakly mixing maps}

Recall that a topological dynamical system $(X,f)$ is called \textit{\emph{(}topologically\emph{)} weakly mixing} if
for any opene $U_1,U_2,V_1,V_2 \in X$ there is a non-negative
integer $n$ such that $U_1 \cap f^{-n}(V_1) \ne \varnothing$ and
$U_2 \cap f^{-n}(V_2) \ne \varnothing$. In other words if its direct product $(X\times X,f\times f)$ is topologically transitive.

\begin{theo} Let $(X,f)$ be a topologically weakly mixing dynamical system.
Then
\begin{itemize}
\item[1.] $\mathbb{L}_{d}=\overline{\mathbb{L}}_{d}=\diam(X)$.
\item[2.] $\mathbb{L}_{r}=\overline{\mathbb{L}}_{r}$.
\item[3.] If, in addition, $(X,f)$ is minimal, then $\mathbb{L}_{r}=\overline{\mathbb{L}}_{r}=\mathbb{L}_{d}=\overline{\mathbb{L}}_{d}=\diam(X).$
\end{itemize}
\end{theo}

\begin{proof} 1. Since a  weakly mixing system is topologically
transitive, from Theorem 3.1 we have
$\mathbb{L}_{d}=\overline{\mathbb{L}}_{d}$.
Since $(X,f)$ be a topologically weakly mixing, also the direct product $(X\times X,f\times f)$ is topologically transitive. So, in every open set in the product, in particular, in the Cartesian square
of every ball $U$ in $X$, there is a transitive point of $(X\times X,f\times f)$, i.e., a pair of points $x,y \in U$. Such pair visits all places in the Cartesian square $X\times X$ infinitely many times. It
means that $\limsup_{n \to \infty}\ d(f^n(x),f^n(y))= \diam(X)=\overline{\mathbb{L}}_{d}$.


2. Let $x\in X$. Since $(X,f)$ is  weakly mixing,
there is a point $z \in X$, such that for any  neighborhood
$G$ of $z$ and any opene $U,V$ in $X$ there exist infinitely many positive
integers $n$, for which $f^n(x) \in G$ and
$f^n(U) \cap V \ne \varnothing$ (\cite{AK}).

By the definition of $\mathbb{L}_{r}$, for the point $z$ and any
(small enough) positive $\delta$ there is a point $y\in X$ and a
positive integer $k$ such that
$d(f^k(y),f^k(z))>\mathbb{L}_{r}-\delta$.

Now, let $U_x$ be a neighborhood of point $x$, let $G_z$ and $V_y$
be open balls of radius $\delta$ centered at points $f^k(z)$ and
$f^k(y)$, respectively. Suppose also $G_z \cap V_y = \varnothing$.
In order to prove the second part of this theorem we will find a point in $U_x$
by using the above property from \cite{AK}. Let $n_0$ be a positive
integer such that $f^{n_0}(x) \in G_z$ and
$f^{n_0}(U_x) \cap V_y \ne \varnothing$. Put
$U_0:= U_x \cap f^{-n_0}(V_y)$. Obviously, $U_0$ is an opene subset of
$U_x$, $\overline{U_0} \subset \overline{U_x}$ and $x \not \in U_0$.
Define inductively opene sets $U_{1},U_{2},...$
and positive integers $n_{i}$ as follows. Let $n_k$ be a (large enough)
positive integer, say $n_k \ge k$, such that $f^{n_k}(x) \in G_z$ and
$f^{n_k}(U_{k-1}) \cap V_y \ne \varnothing$. Define
$U_k:= U_{k-1} \cap f^{-n_k}(V_y)$. It is clear that $U_k$ is an
opene subset in $X$ and $U_i \subset U_{i-1}$ for any $i \ge 1$. Hence
$\overline{U_0} \supset \overline{U_1} \supset \overline{U_2} \supset ...$
If $u$ is a point of the nonempty intersection $\cap_{i} \overline{U_i}$,
then for any natural $i$ we have $f^{n_i}(u) \in \overline V_y$ and
$f^{n_i}(x) \in G_z$. Therefore,
$\limsup_{n \to \infty} d(f^n(x),f^n(u)) \ge \mathbb{L}_{r}-3\delta$.
Since $\delta>0$ is arbitrary, we have
$\mathbb{L}_{r}=\overline{\mathbb{L}}_{r}$.

3. Cases 1 and 2 imply that it is sufficient to prove
$\mathbb{L}_{r}=\diam(X)$. Let $x \in X$ and let $U_x$ be a
neighborhood of $x$. There are two opene (infinite) sets $V_x$ and $V_y$ in
$X$ and a positive (small enough) number $\delta$ such that the
distance between $V_x$ and $V_y$ is large or equal to
$\diam(X)-\delta$.

As we have mentioned before, since TDS $(X,f)$ is minimal, any
point of $X$ is uniformly recurrent. In particular, it means that
$n_f(x,V_x)$ is a syndetic subset of $\mathbb{Z}_{+}$. On the
other hand $(X,f)$ is also a topologically weakly mixing dynamical
system. And again it means that $n_f(U,V_y)$ is a thick subset of
$\mathbb{Z}_{+}$. Hence $n_f(x,V_x) \cap n_f(U,V_y) \ne
\varnothing$ and therefore there exist a point $y \in U$ and a
positive integer $k \in n_f(x,V_x) \cap n_f(U,V_y)$ such that
$f^k(x) \in V_x$ and $f^k(y) \in V_y$. So, $d(f^k(x), f^k(y)) \ge
\diam(X)-\delta$. Since $\delta>0$ was arbitrary, we get
$\mathbb{L}_{r}=\diam(X)$.
\end{proof}

\section {Concluding remarks}

Firstly, let us remark that there are topologically weakly mixing (even topologically mixing) systems for which $\mathbb{L}_{r}=\overline{\mathbb{L}}_{r}=\diam(X)/2$. For instance, the continuous interval map $g:[0,1]\to [0,1]$, where $g(x)=3((x-1/3)-|x-1/3|+|x-2/3|)$, is topologically mixing, one of its fixed points is $1/2$, therefore clearly that $\mathbb{L}_{r}=1/2$.

Also (as we will see in Proposition 5.1 below) there are  dynamical systems for which $\mathbb{L}_{r}=2\overline{\mathbb{L}}_{r}$, but it is still an open question for topologically transitive maps (non-minimal by Theorem 3.2).

Two more open questions:
\begin{itemize}
\item[1.] Does there exist a non-transitive  dynamical system $(X,f)$ for which $\mathbb{L}_{d} > \overline{\mathbb{L}}_{d}$ and/or $\mathbb{L}_{r} > \overline{\mathbb{L}}_{d}$ ?
\item[2.] Does there exist a minimal  dynamical system $(X,f)$ for which  $\mathbb{L}_{d} > \mathbb{L}_{r}$ ?
\end{itemize}

\begin{prop} There exists a topological dynamical system $(X,f)$  for which $\mathbb{L}_{r}=2\overline{\mathbb{L}}_{r}$.
\end{prop}

\begin{proof} We define the space $X$  as a compact surface in
$\mathbb R^3$ which is homeomorphic to a two-dimensional disk in
$\mathbb R^2$. More precisely, the cylindric coordinates of a
point $(x,y,z)\in X$ have the form $(r,\varphi,z)$, where
$r=\sqrt{x^2+y^2}$ and $\varphi$ is an angle, for which
$x=r\cos\varphi$ and $y=r\sin\varphi$. In other words,
$(r,\varphi)$ are the polar coordinates of $(x,y)$, and $z$
remains unchanged. Let $h(r)=8r(1-r)$. Now, define $X$ as a set of
points with cylindric coordinates $(r,\varphi,h(r))$, where $0 \le
r \le 1, \varphi \in \mathbb R$, and  let the Euclidian metric (in
$\mathbb R^3$) $d$ be the metric on $X$.

Now we define a continuous map $f$ from $X$ to itself as follows
$f: (r,\varphi,h(r))\to (g(r),2\varphi,h(g(r)))$, where $g(x)$ is a
continuous map $[0,1] \to [0,1]$ with $g(0)=0$, $g(1)=1$ and $g(x)>x$
for all $x \in (0,1)$. From this properties one can easily deduce
that $\lim_{n \to \infty} g^n(x)=1$ for any $x \in (0,1]$. For example,
let $g(x)=2x-x^2$.

 Let $p \in X$ and $U$ be
a neighborhood of $p$. If $p \ne (0,0,0)$, then for any $\delta>0$
there are $n \in \mathbb N$ and $q \in U$ such that
$d(f^n(p),f^n(q))>2-\delta$. If $p=(0,0,0)$, then there are $n \in
\mathbb N$ and $q \in U$, for which $f^n(q)$ lies on a
circumference of $X$ with the center $(0,0,2)$ (in $\mathbb R^3$)
and the radius $\frac{1}{2}$. For these $n$ and $q$ we have
$d(f^n(p),f^n(q))>2$ and so $\mathbb{L}_{r} \ge 2$.

Now, let $p=(0,0,0)$. The equality $\lim_{n \to \infty}
d(f^n(p),f^n(q))=1$ holds for any $q \ne p$. So
$\overline{\mathbb{L}}_{r} \le 1$. Since $\mathbb{L}_{r} \le
2\overline{\mathbb{L}}_{r}$ (by Proposition 2.1), it gives
$\mathbb{L}_{r}=2\overline{\mathbb{L}}_{r}$.
\end{proof}

The idea of introducing and studying the Lyapunov numbers is derived from the following:

1. If some practical assumption holds for the behavior of a particular system, for example, a physical object, we need to know how far we can go wrong in calculations, if we mean to predict the evolution of the system over a quite long term. Only knowing that there could exist  errors in the calculations of the future behavior of a system is not that useful, since from the practical point of view, the existence of errors in calculations of almost all natural systems (as a result of inaccurate initial data) is a well-known fact. So, quantitative analysis of sensitivity that determines to what extent one's calculations are accurate is of great interest. Comparison of different Lyapunov numbers (the ones which are determined by the upper limit and the ones without limit) demonstrates that errors in calculations cannot disappear (decrease) during passing of time. That is, we cannot expect that, for example, after $10000$ or $1000000$ steps the accuracy of our prediction increase significantly (which seems commonsensical).

2. According to the Auslander theorem, one of the most important
theorems in topological dynamics, any proximal cell (i.e.,
$\Prox_f(x):=\{y\in X: \liminf_{n \to \infty} d(f^n(x),f^n(y))=0\}$)
contains a minimal point \cite{Au}. This implies, in particular, that
a distal point is always minimal. It should be noted that, if $(X,f)$
is a weak mixing dynamical system then for every $x \in X$, the
proximal cell $\Prox_f(x)$ is dense in $X$ \cite{AK}. What about this
property for the sensitive topologically transitive systems, in
particular, for the Devaney systems (i.e., topologically transitive
with a dense set of periodic points systems)? There is a direct connection between this question and the following one: When does $\mathbb{L}_{r}=\overline{\mathbb{L}}_{r}$
hold for a sensitive topologically transitive system?

\end{document}